\documentclass[12pt]{article}
\usepackage{amsfonts}
\usepackage{latexsym}
\usepackage{cite}
\usepackage{amsmath,amsfonts,latexsym,amssymb}
\usepackage[mathscr]{eucal}
\usepackage{cases}
\usepackage{amsthm}

\usepackage[bf,small]{caption2}
\usepackage{float}
\usepackage{graphicx}
\usepackage{amsmath}
\usepackage{amssymb}
\usepackage[all]{xy}

\newtheorem{theorem}{theorem}[section]
\newtheorem{thm}[theorem]{Theorem}
\newtheorem{lem}[theorem]{Lemma}

\newtheorem{rmk}[theorem]{Remark}

\begin{document}

\title{\textbf{Applying TQFT to count regular coverings of Seifert 3-manifolds}}
\author{\Large Haimiao Chen
\date{}
\footnote{Email: \emph{chenhm@math.pku.edu.cn}}\\
\normalsize \em{Department of mathematics, Peking University, Beijing, China}}
\maketitle

\begin{abstract}
I give a formula for computing the number of regular $\Gamma$-coverings of closed orientable Seifert 3-manifolds,
for a given finite group $\Gamma$. The number is computed using a 3d TQFT with finite gauge group, through
a cut-and-glue process.
\end{abstract}

\section{Introduction}
The purpose of this article is to count regular coverings of closed orientable Seifert 3-manifolds, with a given finite
covering group $\Gamma$; the problem is the same as counting homomorphisms from the fundamental groups to $\Gamma$.

The main results are the formulae in Theorem \ref{thm:main1} and \ref{thm:main2}. They give an answer in terms of the
conjugacy classes of $\Gamma$, the centralizers of elements of these classes and their characters.

A \emph{Seifert 3-manifold} is a compact 3-manifold together with a decomposition into a disjoint union of circles
(called \emph{fibers}) such that, each fiber has a tubular neighborhood that is the mapping torus of an automorphism of a
disk given by rotation by an angle of $2\pi b/a$ for a pair of coprime integers $(a,b)$ with $a>0$. A fiber with $a=1$
(resp. $a>1$) is called \emph{ordinary} (resp. \emph{exceptional}). The set of fibers forms a 2-dimensional orbifold called the
\emph{base-surface}. There are two types of connected closed orientable Seifert 3-manifolds according to whether the
base-surface is orientable or not.

Seifert 3-manifolds form an important class of 3-manifolds. Most ``small" 3-manifolds are Seifert manifolds, and they
account for all compact oriented manifolds in 6 of 8 Thurston geometries of the Geometrization Conjecture. As
interesting examples of Seifert 3-manifolds, there are Brieskorn complete intersections (see \cite{Seifert}) which
include homology spheres as a subclass, the complements of torus knots in $S^{3}$ (see \cite{Low} Page 28), and so on.

The problem of enumerating finite-fold coverings of manifolds has been studied extensively in the past decades, especially
in the last twenty years. For example, \cite{RS} gave formulae for regular coverings of surfaces,
both orientable and non-orientable, both with and without boundary. There was much work on the realizability
of branched coverings of surfaces, see, for instance, \cite{Reali2}, \cite{Reali}; the question is then whether the
number of coverings of a certain kind is zero or not. \cite{enum} counted the numbers of homomorphisms from the fundamental groups 
of circle bundles over surfaces to permutation groups, so it actually enables us to count ordinary unbranched coverings 
of circle bundles over surfaces.

Generally speaking, the enumeration of isomorphism classes of (not necessarily connected) coverings of a space $M$ is
reduced to that of homomorphisms from $\pi_{1}(M)$ to a finite group. Namely, $n$-sheeted ordinary coverings of $M$
correspond bijectively to homomorphisms $\pi_{1}(M)\rightarrow S_{n}$ ($S_{n}$ is the permutation group on $n$ letters);
for a finite group $\Gamma$, regular $\Gamma$-coverings of $M$ correspond bijectively to homomorphisms
$\pi_{1}(M)\rightarrow\Gamma$. For more details see \cite{AT}.

While there has been much work on enumerations in 2-dimensional topology, there is, besides \cite{enum}, little on
3-dimensional topology, due to the complexity of 3-manifold groups.

Our results make some contributions to 3-dimensional enumeration. The approach uses TQFT with finite gauge group.
Just as V.Turaev (\cite{Turaev1},\cite{Turaev2}) applied 2d TQFT to count representations of surface groups, we
apply 3d TQFT to count representations of the fundamental groups of Seifert 3-manifolds. Compared with \cite{enum},
we go further in two directions: the manifolds are allowed to have exceptional fibers, so they can be more
complicated than just circle bundles; the target group may be any finite group, not just a permutation group.

The article is organized as follows. In Section 2, some basic notions and facts of TQFT are recalled. Section
3 and 4 are devoted to deriving formulae for enumeration. Some concrete computations are done in Section 5. The
last section contains some remarks.

\vspace{4mm}

\textbf{Notations.}

$\Sigma_{g}$: the orientable closed surface of genus $g$.

$\Sigma_{g;p,q}$: the orientable surface of genus $g$, with boundary $(p+q)$ circles, $p$ of which are oriented
negatively, and the other $q$ positively.

$P$: the pair of pants

$ST$: the solid torus $S^{1}\times D$

$\#A$: the cardinality of a finite set $A$.

$C(x)$: the centralizer of an element $x$ in a group.

\section{TQFT with finite gauge group}

In this section, we recall some notions and facts on TQFT. References are \cite{TC}, \cite{DW}, \cite{TQFT}, \cite{Freed}. 
Assume all manifolds are compact, smooth and oriented, and all maps are orientation-preserving
diffeomorphisms.

\subsection{Axioms of TQFT}
For an integer $l$, as proposed by Freed (see \cite{Freed}), an \emph{$(l+1)$-dimensional topological quantum field
theory} ($(l+1)$d TQFT for short) is an assignment $Z$, assigning to each $l$-dimensional closed manifold $Y$ a
finite-dimensional Hermitian inner product space $Z(Y)$, with $Z(\emptyset)=\mathbb{C}$ (equipped with the
standard inner product), and to each $(l+1)$-dimensional manifold $X$ an element of the vector space $Z(\partial X)$,
such that the following holds:

(a) (Functoriality) Every map $f:Y\rightarrow Y'$ induces an isometry
$$f_{\ast}:Z(Y)\rightarrow Z(Y');$$
and for $f:X\rightarrow X',f':X'\rightarrow X''$,
$$f'_{\ast}\circ f_{\ast}=(f'\circ f)_{\ast}.$$

For every map $F:X\rightarrow X'$, one has
$$(\partial F)_{\ast}(Z(X))=Z(X').$$

(b) (Orientation) There is a natural isometry
$$Z(-Y)\cong\overline{Z(Y)},$$
($-Y$ has the same underlying manifold as $Y$ but the opposite orientation) through which
$$Z(-X)=\overline{Z(X)}.$$

(c) (Multiplicativity) There is a natural isometry
$$Z(Y_{1}\sqcup Y_{2})\cong Z(Y_{1})\otimes Z(Y_{2}),$$
through which
$$Z(X_{1}\sqcup X_{2})=Z(X_{1})\otimes Z(X_{2}).$$

(d) (Gluing) If $Y^{l}\hookrightarrow X^{l+1}$ is a submanifold and $X^{cut}$ is the manifold obtained by
cutting $X$ along $Y$, so that $\partial X^{cut}=\partial X\sqcup Y\sqcup -Y$, then
$$Z(X)=Tr_{Y}(Z(X^{cut})),$$
where
$$Tr_{Y}:Z(\partial X)\otimes Z(Y)\otimes\overline{Z(Y)}\rightarrow Z(\partial X)$$
is the contraction using the inner product on $Z(Y)$.

\vspace{4mm}

It follows from (b),(c) that when $W$ is a cobordism from $M_{1}$ to $M_{2}$, that is, $\partial W=-M_{1}\sqcup M_{2}$,
then
$$Z(W)\in\overline{Z(M_{1})}\otimes Z(M_{2})\cong\hom(Z(M_{1}),Z(M_{2})),$$
hence $Z(W)$ can be identified with a linear map $Z(M_{1})\rightarrow Z(M_{2})$. The axioms (a),(d) then tell us that
$Z$ is also a functor from the category of $(l+1)$-dimensional cobordisms to that of inner product spaces.

Also, (b) says that $Z(-W):\overline{Z(M_{2})}\rightarrow\overline{Z(M_{1})}$ is the dual map of $Z(W)$.
\begin{rmk}
$Z(X),Z(Y)$ are called the \emph{path-integrals} of $X,Y$, respectively.
\end{rmk}
\begin{rmk}
The gluing axiom (d) also means that, when $\partial X_{1}=-Y,\partial X_{2}=Y$, and $X$ is obtained by gluing
$X_{1},X_{2}$ along $Y$ via a map $f:Y\rightarrow Y$, i.e., $X=(X_{1}\sqcup X_{2})/(y\in\partial X_{1})
\sim (f(y)\in\partial X_{2})$, then
$$Z(X)=Tr_{Y}(Z(X_{2})\otimes f_{\ast}Z(X_{1})).$$
\end{rmk}

\subsection{TQFT with finite gauge group}

Given a finite group $\Gamma$ and an integer $l$, there is a general method to construct an $(l+1)$d TQFT as follows;
see \cite{TQFT}, \cite{Freed}.

For a manifold $M$, let $\mathfrak{C}_{M}$ denote the groupoid of principal $\Gamma$-bundles over $M$ with morphisms
being the bundle morphisms covering the identity on $M$. Let $\overline{\mathfrak{C}}_{M}$ be the set of equivalence classes
of these bundles.

For $P\in\mathfrak{C}_{M}$, setting
$$\mu(P)=\frac{1}{\#\text{Aut}(P)}$$
defines a measure on $\mathfrak{C}_{M}$,
and it descends to $\overline{\mathfrak{C}}_{M}$, by $\mu([P])=\frac{1}{\#\text{Aut}(P)}$ for any $P\in
[P]\in\overline{\mathfrak{C}}_{M}$.

For $Q\in\mathfrak{C}_{\partial M}$, let $\mathfrak{C}_{M}(Q)$ be the groupoid of bundles $P\in\mathfrak{C}_{M}$
such that $\partial P=Q$, with morphisms being the bundle morphisms whose restriction to $Q$ is the identity.

Let $\overline{\mathfrak{C}}_{M}(Q)$ be the set of equivalence classes. As above, one can define a measure on
$\overline{\mathfrak{C}}_{X}(Q)$ by
$$\mu_{Q}([P])=\frac{1}{\#\text{Aut}_{Q}(P)},$$
where $\text{Aut}_{Q}(P)$ is the group of automorphisms of $P$ fixing $Q$.

For each closed $l$-dimensional manifold $Y$, set
$$Z(Y)=\text{Map}(\overline{\mathfrak{C}}_{Y},\mathbb{C}),$$
and equip it with the inner product
$$(f,g)=\sum\limits_{[Q]\in\overline{\mathfrak{C}}_{Y}}\mu([Q])f([Q])\overline{g([Q])}.$$

For each $(l+1)$-dimensional manifold $X$, set
$$Z(X)([Q])=\sum\limits_{[P]\in\overline{\mathfrak{C}}_{X}(Q)}\mu_{Q}([P])\in\mathbb{C}$$ for all
$[Q]\in\overline{\mathfrak{C}}_{\partial X}$, where $Q$ is any representative of $[Q]$. Thus $Z(X)\in Z(\partial X)$.

\vspace{4mm}

A more concrete description was given in Section 5 of \cite{Freed}. Here we state it as a theorem.
\begin{thm} \label{thm:DW}
Suppose $X^{l+1},Y^{l}$ are connected.
\begin{align}
Z(Y)=\text{Map}(\hom(\pi_{1}(Y),\Gamma)/\Gamma,\mathbb{C}),
\end{align}
where $\Gamma$ acts by conjugation; for $\gamma:\pi_{1}(Y)\rightarrow\Gamma$, $\mu([\gamma])=\frac{1}{\#C(\gamma)}$,
with $C(\gamma)$ being the centralizer of $im\gamma\subset\Gamma$.

If $X$ is closed,
\begin{align}
Z(X)=\frac{1}{\#\Gamma}\cdot\#\hom(\pi_{1}(X),\Gamma). \label{eq:Z(closed)}
\end{align}
And if $\partial X\neq\emptyset$, for $\beta\in\hom(\pi_{1}(\partial X),\Gamma)$,
\begin{align}
Z(X)([\beta])=\#(\iota^{\ast})^{-1}(\beta), \label{eq:Z(open)}
\end{align}
where $\iota^{\ast}:\hom(\pi_{1}(X),\Gamma)\rightarrow\hom(\pi_{1}(\partial X),\Gamma)$ is the restriction map.
\end{thm}
\begin{rmk}
Such a TQFT is known as (untwisted) \emph{Dijkgraaf-Witten theory}, named by the authors of \cite{DW}, who first proposed it.
\end{rmk}

\subsection{(2+1)-dimensional DW theory}

When $l=2$, $E:=Z(\Sigma_{1})$ becomes the
vector space of maps
\begin{align}\theta:\{(x,h)\in\Gamma\times\Gamma|xh=hx\}\rightarrow\mathbb{C}
\end{align}
satisfying
\begin{align}
\theta(uxu^{-1},uhu^{-1})=\theta(x,h),\hspace{5mm}\forall u\in\Gamma;
\end{align}
and the inner product is given by
\begin{align}
(\theta_{1},\theta_{2})=\frac{1}{\Gamma}\cdot\sum\limits_{x,h}\theta_{1}(x,h)\overline{\theta_{2}(x,h)}.
\end{align}
By Lemma 5.4 of \cite{Freed}, $E$ has a canonical orthonomal basis $\{\chi_{i}|i\in \Lambda\}$, where
$\Lambda=\{i=(c,\rho)\}$, with $c$ a conjugacy class of $\Gamma$ and $\rho$ an irreducible character of $C(x)$,
the centralizer of $x$ for an arbitrary choice of $x\in c$. We say that $\chi_{i}$ is \emph{supported in the
conjugacy class $c$} and denote
\begin{align}
c=\text{supp}(\chi_{i}).
\end{align}
The Explicit expression of $\chi_{i}$ is
\begin{align}
\chi_{(c,\rho)}(x,h)=\rho(h), \forall x\in c, h\in C(x); \hspace{4mm} \chi_{(c,\rho)}(x,h)=0, \forall x\notin c.
\end{align}

Define
\begin{align}
\dim\chi_{(c,\rho)}=\#c\cdot\rho(e). \label{defn:dim}
\end{align}

\vspace{4mm}

Let $P$ be the pair of pants (see Figure 1). Since $\partial (P\times S^{1})=-\Sigma_{1}\sqcup-\Sigma_{1}\sqcup
\Sigma_{1}$, we have
$$Z(P\times S^{1})\in \overline{E}\otimes\overline{E}\otimes E\cong\hom(E\otimes E,E),$$
so it gives a product on $E$,
\begin{align}
m:&E\otimes E\rightarrow E, \nonumber    \\
m(\theta\otimes\vartheta)(x,h)&=\sum\limits_{x_{1}x_{2}=x}\theta(x_{1},h)\vartheta(x_{2},h).\label{eq}
\end{align}
The proof of formula (\ref{eq}) is similar to the proof of Proposition 5.17 of \cite{Freed}.
\begin{figure}[h]
  \centering
  \includegraphics[width=0.6\textwidth]{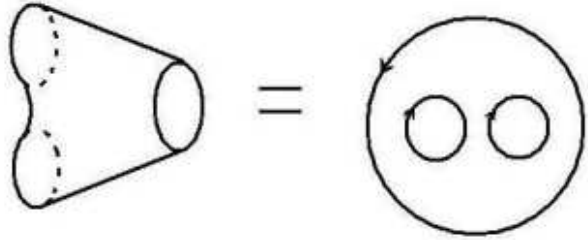}\\
  \caption{the pair of pants}\label{}
\end{figure}

The mapping class group of $\Sigma_{1}$, which is isomorphic to $SL(2,\mathbb{Z})$ (see \cite{Topof3mfd},
Page 22), acts on $E$ by
\begin{align}
\left(\left(\begin{array}{cc}a & b \\ c& d \end{array}\right)_{\ast}\theta\right)(x,h)=\theta(x^{a}h^{b},x^{c}h^{d}).
\end{align}

For the generators
$T=\left(\begin{array}{cc}
1 & 0 \\
1 & 1
\end{array}\right),
S=\left(\begin{array}{cc}
0 & -1 \\
1 & 0
\end{array}\right)$,
the expressions are (see \cite{Freed}, Proposition 5.8)
\begin{align}
T_{\ast}\chi_{i}&=\kappa_{i}\chi_{i},\\
S_{\ast}\chi_{i}&=\sum\limits_{j\in\Lambda}s_{i}^{j}\chi_{j},
\end{align}
where
\begin{align}
\kappa_{i}=\chi_{i}(x,x)/\chi_{i}(x,e),
\end{align}
with $x\in \text{supp}(\chi_{i})$ chosen arrbitrarily, and
\begin{align}
s_{i}^{j}=(S_{\ast}\chi_{i},\chi_{j})=\frac{1}{\#\Gamma}\cdot\sum\limits_{x,h}\chi_{i}(h^{-1},x)\overline{\chi_{j}(x,h)}.
\end{align}
In particular,
\begin{align}
s_{0}^{i}=\frac{\dim\chi_{i}}{\#\Gamma}.
\end{align}

Suppose $m(\chi_{i}\otimes\chi_{j})=\sum\limits_{k}N_{ij}^{k}\chi_{k}$. By Proposition 3.1.12 of \cite{TC}, if we
define the matrices $s,N_{i},D_{i}$ respectively by
$$s_{ij}=s_{i}^{j},\hspace{4mm} (N_{i})_{jk}=N_{ij}^{k},\hspace{4mm} (D_{i})_{ab}=\delta_{ab}\frac{s_{i}^{a}}{s_{0}^{a}},$$
then $sN_{i}s^{-1}=D_{i}$. So setting
\begin{align}
\tau_{i}=S^{-1}_{\ast}\chi_{i}
\end{align}
will diagonalize the product:
\begin{align}
m(\tau_{i}\otimes\tau_{j})=\frac{\delta_{ij}}{s_{0}^{i}}\tau_{i}.
\end{align}
Since $S_{\ast}^{-1}$ is unitary, $\{\tau_{i}|\in\Lambda\}$ is still an orthonormal basis of $E$.

It is easy to see, by dualizing, that $Z(-P\times S^{1})$ gives a coproduct
\begin{align}
&m':E\rightarrow E\otimes E, \nonumber \\
&m'(\tau_{i})=\frac{1}{s_{0}^{i}}\tau_{i}\otimes\tau_{i}.
\end{align}
\vspace{4mm}

For the solid torus which is denoted $ST$, let $h$ indicate the longitude, and $x$ the meridian.

Note that the homomorphism $\pi_{1}(T)\rightarrow\pi_{1}(ST)$ induced by the inclusion $T\hookrightarrow ST$ is
a surjection whose kernel is generated by $x$, so we have
$$Z(ST)(x,h)=\delta_{x,e},\hspace{4mm}(S_{\ast}Z(ST))(x,h)=\delta_{h,e},$$
and
\begin{align}
Z(ST)=\sum\limits_{i\in\Lambda}s_{0}^{i}\tau_{i},
\end{align}
since $(Z(ST),\tau_{i})=(S_{\ast}Z(ST),\chi_{i})=s_{0}^{i}$.

Regarding $Z(-ST)$ as a morphism $E\rightarrow \mathbb{C}$,
\begin{align}
Z(-ST)(\tau_{i})=(\tau_{i},Z(ST))=\overline{s_{0}^{i}}=s_{0}^{i}.
\end{align}

\section{Formula for Seifert 3-manifolds I: \\ orientable base-surfaces}

The strategy for determining $\#\hom(\pi_{1}(M),\Gamma)$ for a 3-manifold $M$ is, relying on Theorem
\ref{thm:DW}, to compute $Z(M)$.

\subsection{Orientable Seifert 3-manifolds}

According to \cite{Topof3mfd}, a closed orientable Seifert 3-manifold $M$ can be obtained as follows. Take a
circle bundle $F$ which is orientable as a manifold, over a surface $R$ with $\partial R=\sqcup_{n}S^{1}$,
so that $\partial F=\sqcup_{n}\Sigma_{1}$, and glue $n$ solid tori onto $F$ along the boundary $\Sigma_{1}$'s,
via diffeomorphisms $f_{j}:\Sigma_{1}\rightarrow\Sigma_{1},j=1,\cdots,n$. The ``closure" of $R$,
$\overline{R}=R\cup(\sqcup_{n}D)$, is the base-surface, and the images of the cores of $ST$,
$S^{1}\times 0\subset ST$, are the exceptional fibers.

When $f_{j}$ lies in the mapping class of $\Sigma_{1}$ represented by
$\left(\begin{array}{cc} a_{j} & b_{j} \\ u_{j} & v_{j} \end{array}\right)$, denote $M$ as
$M(\overline{R};(a_{1},b_{1}),\cdots,(a_{n},b_{n}))$.

In this section we assume $\overline{R}$ is orientable, $R=\Sigma_{g;n,0}$. Then $F$ is diffeomorphic to
$\Sigma_{g;n,0}\times S^{1}$. Denote $M$ as $M_{O}(g;(a_{1},b_{1}),\cdots,(a_{n},b_{n}))$.

\subsection{Computing $Z(\Sigma_{g;p,q}\times S^{1})$}

Since $\Sigma_{1;1,1}=-P\cup_{S^{1}}P$, $Z(\Sigma_{1;1,1}\times S^{1})$ is the composite
\begin{align}
E\xrightarrow[]{m'}E\otimes E\xrightarrow[]{m} E, \hspace{4mm}\tau_{i}\mapsto(s_{0}^{i})^{-2}\tau_{i}.
\end{align}

In general, since $\Sigma_{g;1,1}$ can be obtained by gluing $g$ $\Sigma_{1;1,1}$'s successively, as shown in
Figure 2(a), we have
\begin{align}
Z(\Sigma_{g;1,1}\times S^{1})=(m\circ m')^{g}:E\rightarrow E,\hspace{4mm}\tau_{i}\mapsto(s_{0}^{i})^{-2g}\tau_{i}.
\end{align}

\begin{figure}[h]
  \centering
  \includegraphics[width=0.5\textwidth]{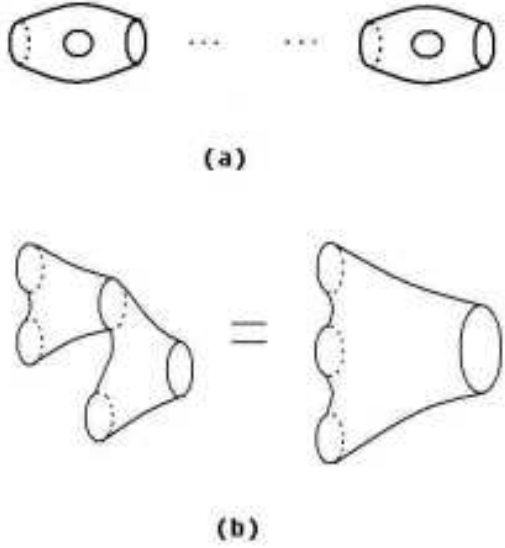}\\
  \caption{(a) $\Sigma_{g;1,1}$ obtained by gluing $\Sigma_{1;1,1}$'s; (b) $\Sigma_{0;p,1}$ obtained by gluing $P$'s}\label{}
\end{figure}

For $p>0$, $\Sigma_{0;p,1}$ can be obtained by gluing $(p-1)$ $P$'s successively, (Figure 2(b) illustrates the
case $p=3$), hence $Z(\Sigma_{0;p,1}\times S^{1}):E^{\otimes p}\rightarrow E$ is equal to
\begin{align}
m &\circ(1\otimes m)\circ\cdots\circ(1\otimes\cdots\otimes 1\otimes m),\nonumber\\
&\tau_{i_{1}}\otimes\cdots\otimes\tau_{i_{p}}\mapsto(s_{0}^{i_{1}})^{1-p}\delta_{i_{1},\cdots,i_{p}}\tau_{i_{1}}.
\end{align}

Dually, for $q>0$, $Z(\Sigma_{0;1,q}\times S^{1}):E\rightarrow E^{\otimes q}$ is equal to
\begin{align}
(1\otimes\cdots\otimes 1\otimes m')\circ\cdots\otimes(1\otimes m')\circ m', \hspace{4mm}
\tau_{i}\mapsto(s_{0}^{i})^{1-q}\tau_{i}^{\otimes q}.
\end{align}

For $p,q>0$, $Z(\Sigma_{g;p,q}\times S^{1}):E^{\otimes p}\rightarrow E^{\otimes q}$ is equal to the composite
\begin{align}
E^{\otimes p}\xrightarrow[]{Z(\Sigma_{0;p,1}\times S^{1})}&E\xrightarrow[]{Z(\Sigma_{g;1,1}\times S^{1})}E\xrightarrow[]
{Z(\Sigma_{0;1,q}\times S^{1})}E^{\otimes q},\nonumber \\
Z(\Sigma_{g;p,q}\times S^{1})(\tau_{i_{1}}\otimes &\cdots\otimes\tau_{i_{p}})=(\frac{1}{s_{0}^{i_{1}}})^{p+q+2g-2}
\delta_{i_{1},\cdots,i_{p}}\tau_{i_{1}}^{\otimes q}.
\end{align}

\subsection{Considering exceptional fibers}

Let $M_{O}'=M_{O}'(g;(a_{1},b_{1}),\cdots,(a_{n},b_{n}))$ be $M_{O}(g;(a_{1},b_{1}),\cdots,(a_{n},b_{n}))$
with an $ST$ deleted. It is obtained by gluing $n$ $ST$'s onto $\Sigma_{g;n,1}\times S^{1}$, using
$f_{j}:\Sigma_{1}\rightarrow\Sigma_{1}, j=1,\cdots,n$.

Since $(f_{j})_{\ast}(Z(ST))(x,h)=\delta_{e,x^{a_{j}}h^{b_{j}}}$, we have
\begin{align}
(f_{j})_{\ast}(Z(ST))=\frac{1}{\#\Gamma}\cdot\sum\limits_{i\in\Lambda}\left(\sum\limits_{xh=hx \atop x^{a_{j}}h^{b_{j}}=e}\overline{\tau_{i}(x,h)}\right)\tau_{i}.
\end{align}
For any pair of coprime integers $(a,b)$, define
\begin{align}
\eta_{i}(a,b)=\sum\limits_{xh=hx \atop x^{a}h^{b}=e}\overline{\tau_{i}(x,h)}.
\end{align}
\begin{lem}
\begin{align}
\eta_{i}(a,b)=\sum\limits_{z\in\Gamma}\overline{\tau_{i}(z^{-b},z^{a})}=\sum\limits_{z\in\Gamma}\chi_{i}(z^{a},z^{-b}).
\end{align}
\end{lem}
\begin{proof}
since $(a,b)=1$, there exists a $k$ such that $(ka-b,\#\Gamma)=1$, (just let $k$ be the product of all prime
factors of $\#\Gamma$ that do not divide $b$). Take $r$ such that $r(ka-b)\equiv 1\pmod{\#\Gamma}$, then whenever
$hx=xh,x^{a}h^{b}=e$, one has $h=h^{r(ka-b)}=h^{rka}x^{ra}=(xh^{k})^{ra}$; let $y=xh^{k}$, we see
$h=y^{ra},x=yh^{-k}=y^{1-rka}=y^{-rb}$. Note that the map
$$\{(x,h)|xh=hx,x^{a}h^{b}=e\}\rightarrow\Gamma,\hspace{5mm} (x,h)\mapsto (xh^{k})^{r},$$
is bijective, with the inverse map given by $z\mapsto (z^{-b},z^{a})$. So the lemma is established.
\end{proof}
\begin{rmk}  \label{rmk:eta}
Suppose $(a,\#\Gamma)=d$. Taking $c$ such that $ac\equiv d\pmod{\#\Gamma}$, we have
\begin{align}
\eta_{i}(a,b)=\sum\limits_{z^{a/d}=y\in\Gamma}\chi_{i}(z^{a},z^{-b})=\sum\limits_{y\in\Gamma}\chi_{i}(y^{d},y^{-bc}).
\end{align}

In particular, when $(a,\#\Gamma)=1$, $\eta_{i}(a,b)=\dim\chi_{i}\cdot\kappa_{i}^{-bc}$.
\end{rmk}

Going on, $Z(M_{O}')$ is the composite
\begin{align}
\mathbb{C}\cong &\mathbb{C}^{\otimes n}\xrightarrow[]{\bigotimes\limits_{j=1}^{n}((f_{j})_{\ast}Z(ST))}E^{\otimes n}\xrightarrow[]{Z(\Sigma_{g;n,1}\times S^{1})}E. \nonumber \\
Z(M_{O}')(1)&=\sum\limits_{i\in\Lambda}\left(\frac{1}{(\#\Gamma)^{n}(s_{0}^{i})^{n+2g-1}}\cdot\prod\limits_{j=1}^{n}
\eta_{i}(a_{j},b_{j})\right)\tau_{i}. \label{eq:Z(M')}
\end{align}

Finally, $Z(M_{O}(g;(a_{1},b_{1}),\cdots,(a_{n},b_{n})))$ is the composite
\begin{align}
\mathbb{C}\xrightarrow[]{Z(M_{O}')} &E\xrightarrow[]{Z(-ST)}\mathbb{C},\nonumber \\
Z(M_{O}(g;(a_{1},b_{1}),\cdots,(a_{n},b_{n})))&=\sum\limits_{i\in\Lambda}\frac{(\#\Gamma)^{2g-2}}{(\dim\chi_{i})^{n+2g-2}}
\cdot\prod\limits_{j=1}^{n}
\eta_{i}(a_{j},b_{j}).
\end{align}
Thus we have the following
\begin{thm} \label{thm:main1}
The number of regular $\Gamma$-coverings of the Seifert 3-manifold $M_{O}(g;(a_{1},b_{1}),\cdots,(a_{n},b_{n}))$ is
\begin{align}
\sum\limits_{i\in\Lambda}\frac{(\#\Gamma)^{2g-1}}{(\dim\chi_{i})^{n+2g-2}}\cdot\prod\limits_{j=1}^{n}
\eta_{i}(a_{j},b_{j}).
\end{align}
\end{thm}
\begin{rmk}
Be careful that, by (\ref{eq:Z(closed)}), $Z(M)$ is the number of homomorphisms $\pi_{1}(M)\rightarrow\Gamma$
divided by $\#\Gamma$.
\end{rmk}

\section{Formula for Seifert 3-manifolds II: \\ non-orientable base-surfaces}

When the base-surface $\overline{R}$ is non-orientable, $\overline{R}$ is the connected sum of some
$\mathbb{R}P^{2}$'s. Denote $\Pi_{g}=\#_{g}\mathbb{R}P^{2}$. Let $\Pi_{g;n}$ be $\Pi_{g}$ with $n$ disks removed.

In the notation of Section 3.1, when $R=\Pi_{g;n}$, denote the Seifert manifold by
$M_{N}(g;(a_{1},b_{1}),\cdots,(a_{n},b_{n}))$.

According to \cite{enum}, over $\Pi_{g;1}$ there is up to isomorphism only one orientable circle bundle $F_{g}$,
whose fundamental group has the presentation
$$\langle x,h,y_{1},\cdots,y_{g}|x=\prod\limits_{j=1}^{g}y_{j}^{2},y_{j}h=hy_{j},h^{2}=e\rangle.$$
Noticing that $\Pi_{g,n}=\Pi_{g,1}\cup_{S^{1}}\Sigma_{0;n,1}$, it is then easy to see that the Seifert manifold $M_{N}(g;(a_{1},b_{1}),\cdots,(a_{n},b_{n}))$
can be obtained by gluing $F_{g}$ with $M'_{O}=M'_{O}(0;(a_{1},b_{1}),\cdots,(a_{n},b_{n}))$ along $\Sigma_{1}$.

Since $\partial F_{g}=\Sigma_{1}$, $Z(F_{g})\in E$. By (\ref{eq:Z(open)}), when $h^{2}\neq e$, $Z(F_{g})(x,h)=0$;
when $h^{2}=e$, $Z(F_{g})(x,h)=\#\{(y_{1},\cdots,y_{g})\in\Gamma^{g}|y_{j}h=hy_{j},\prod\limits_{j=1}^{g}y_{j}^{2}=x\}$.
Thus
\begin{align}
(Z(F_{g}),\tau_{i})&=(S_{\ast}Z(F_{g}),\chi_{i})\nonumber \\
&=\frac{1}{\#\Gamma}\cdot\sum\limits_{xh=xh \atop x^{2}=e}\#\{(y_{1},\cdots,y_{g})\in\Gamma^{g}|h\prod
\limits_{j=1}^{g}y_{j}^{2}=e,y_{j}\in C(x)\}\overline{\chi_{i}(x,h)}.
\end{align}
But by \cite{RS} Theorem 4,
\begin{align}
&\#\{(y_{1},\cdots,y_{g})\in\Gamma^{g}|h\prod\limits_{j=1}^{g}y_{j}^{2}=e,y_{j}\in C(x)\}\nonumber\\
=&(\#C(x))^{g-1}\sum\limits_{\chi}c_{\chi}^{g}\chi(e)^{1-g}\chi(h),
\end{align}
where $\chi$ ranges over the irreducible characters of $C(x)$, and $c_{\chi}$ is the Frobenius-Schur indicator
(see \cite{RS}), which is defined for any character $\rho$ of a group $G$,
\begin{align}
c_{\rho}=\frac{1}{\#G}\cdot\sum\limits_{y\in G}\rho(y^{2}).
\end{align}
Here for $j\in\Lambda$, we modify it to define by taking any $x\in\text{supp}(\chi_{j})$,
\begin{align}
\tilde{c}_{j}=\frac{\delta_{e,x^{2}}}{\#C(x)}\cdot\sum\limits_{h\in C(x)}\chi_{j}(x,h^{2}).
\end{align}

Then
\begin{align}
(Z(F_{g}),\tau_{i})
&=\frac{1}{\#\Gamma}\sum\limits_{x^{2}=e}\sum\limits_{j\in\Lambda\atop x\in\text{supp}(\chi_{j})}\sum\limits_{h\in C(x)}(\#C(x))^{g-1}c_{\chi_{j}}^{g}\chi_{j}(x,e)^{1-g}\chi_{j}(x,h)\overline{\chi_{i}(x,h)} \nonumber\\
&=\frac{\tilde{c}_{i}^{g}}{\#\Gamma}\cdot\sum\limits_{x\in\text{supp}(\chi_{i})}\#C(x)(\frac{\#C(x)}{\chi_{i}(x,e)})^{g-1}
=(\frac{\#\Gamma}{\dim\chi_{i}})^{g-1}\cdot\tilde{c}_{i}^{g}, \label{eq:Z(Fg)}
\end{align}
where we have used the orthogonality relation $\sum\limits_{h\in C(x)}\chi_{j}(x,h)\overline{\chi_{i}(x,h)}\\
=\#C(x)\cdot\delta_{i,j}$, the identity $\#\text{supp}(\chi_{i})\cdot\#C(x)=\#\Gamma$ and (\ref{defn:dim}).

Now $Z(M_{N}(g;(a_{1},b_{1}),\cdots,(a_{n},b_{n})))$ is equal to the composite $$\mathbb{C}\xrightarrow[]{Z(M'_{O})}E\xrightarrow[]{(\cdot,Z(F_{g}))}\mathbb{C}.$$

By (\ref{eq:Z(M')}),(\ref{eq:Z(Fg)}), we obtain
\begin{align}
Z(M_{N}(g;(a_{1},b_{1}),\cdots,(a_{n},b_{n})))=\sum\limits_{i\in\Lambda}\frac{(\#\Gamma)^{g-2}\tilde{c}_{i}^{g}}
{(\dim\chi_{i})^{n+g-2}}\prod\limits_{j=1}^{n}\eta_{i}(a_{j},b_{j}).
\end{align}
Thus the enumeration for regular coverings of Seifert 3-manifolds with non-orientable base-surfaces is completed:
\begin{thm} \label{thm:main2}
The number of regular $\Gamma$-coverings of the Seifert 3-manifold $M_{N}(g;(a_{1},b_{1}),\cdots,(a_{n},b_{n}))$ is
\begin{align}
\sum\limits_{i\in\Lambda}\frac{(\#\Gamma)^{g-1}\tilde{c}_{i}^{g}}{(\dim\chi_{i})^{n+g-2}}\prod
\limits_{j=1}^{n}\eta_{i}(a_{j},b_{j}).
\end{align}
\end{thm}

\section{Computations}

Take $\Gamma=A_{5}$, the group of even permutations on 5 elements. It is the smallest nontrivial simple group.

In this section we compute the number of regular $A_{5}$-coverings of any orientable closed Seifert 3-manifold
with orientable base-surface.

\subsection{Facts on representation theory}

The facts on $A_{5}$ which we recall below have been taken from \cite{Artin} (Page 324).

There are 5 conjugacy classes:
$$[1],\hspace{4mm} [\alpha],\hspace{4mm} [\beta],\hspace{4mm} [\gamma],\hspace{4mm} [\gamma^{2}],$$
where $$1=(1),\hspace{4mm }\alpha=(12)(34),\hspace{4mm} \beta=(123),\hspace{4mm} \gamma=(12345),$$
and $[x]$ stands for the conjugacy class represented by $x$.

It is known that
\begin{align}
(\#[1],\#[\alpha],\#[\beta],\#[\gamma],\#[\gamma^{2}])=(1,15,20,12,12).
\end{align}
The corresponding centralizers are
\begin{align*}
C(1)=A_{5}&, \hspace{4mm}
C(\alpha)=\{(1),(12)(34),(13)(24),(14)(23)\}\cong\mathbb{Z}/2\mathbb{Z}\times\mathbb{Z}/2\mathbb{Z}, \\
&C(\beta)=\langle\beta\rangle\cong\mathbb{Z}/3\mathbb{Z}, \hspace{4mm}
C(\gamma)=C(\gamma^{2})=\langle\gamma\rangle\cong\mathbb{Z}/5\mathbb{Z}.
\end{align*}
The irreducible characters of $C(1)=A_{5}$ are given by
\begin{align}
&\rho^{1}_{1}(1,\alpha,\beta,\gamma,\gamma^{2})=(1,1,1,1,1),\\ &\rho^{1}_{2}(1,\alpha,\beta,\gamma,\gamma^{2})=(3,-1,0,\frac{1+\sqrt{5}}{2},\frac{1-\sqrt{5}}{2}),\\ &\rho^{1}_{3}(1,\alpha,\beta,\gamma,\gamma^{2})=(3,-1,0,\frac{1-\sqrt{5}}{2},\frac{1+\sqrt{5}}{2}),\\
&\rho^{1}_{4}(1,\alpha,\beta,\gamma,\gamma^{2})=(4,0,1,-1,-1), \\
&\rho^{1}_{5}(1,\alpha,\beta,\gamma,\gamma^{2})=(5,1,-1,0,0).
\end{align}
The other four centralizers are all abelian, and their irreducible characters are easy to find.

For $C(\alpha)$, set
\begin{align*}
&\rho^{\alpha}_{1}\equiv 1,\hspace{6mm} \rho^{\alpha}_{2}(1,\alpha,(13)(24),(14)(23))=(1,-1,1,-1), \\ \rho^{\alpha}_{3}&(1,\alpha,(13)(24),(14)(23))=(1,1,-1,-1),\hspace{6mm}\rho^{\alpha}_{4}=\rho^{\alpha}_{2}\rho^{\alpha}_{3}.
\end{align*}
For $C(\beta)$, set
\begin{align*}
\rho^{\beta}_{j}(\beta)=e^{2(j-1)\pi i/3}, \hspace{4mm} 1\leqslant j\leqslant 3.
\end{align*}
For $C(\gamma)$, set
\begin{align*}
\rho^{\gamma}_{j}(\gamma)=e^{2(j-1)\pi i/5}, \hspace{4mm} 1\leqslant j\leqslant 5.
\end{align*}
For $C(\gamma^{2})$, set
\begin{align*}
\rho^{\gamma^{2}}_{j}=\rho^{\gamma}_{j}, \hspace{4mm} 1\leqslant j\leqslant 5.
\end{align*}
These are all the irreducible characters.

At last, to be convenient, let
\begin{align}
\Lambda=\{([x],\rho)|x=1,\alpha,\beta,\gamma,\gamma^{2},\rho=\rho^{x}_{j}\}.
\end{align}
And for $i=([x],\rho)\in\Lambda$, define $\chi_{i}\in E$ by putting $\chi_{i}(y,g)=\rho(g)$ if $y$ is
conjugate to $x$ and $g\in C(y)$, and putting $\chi_{i}(y,g)=0$ otherwise.

\subsection{Evaluating $\eta$}

Now begin to evaluate $\eta_{i}(a,b)$ for any pair of coprime integers $(a,b)$. Remark \ref{rmk:eta}
will be referred to repeatedly.

For $1\leqslant j\leqslant 5$ set
\begin{align}
&\omega_{j}(2)=\sum\limits_{h\in [\alpha]}\rho^{1}_{j}(h)=15\rho^{1}_{j}(\alpha),\\
&\omega_{j}(3)=\sum\limits_{h\in [\beta]}\rho^{1}_{j}(h)=20\rho^{1}_{j}(\beta),\\
&\omega_{j}(5)=\sum\limits_{h\in [\gamma]\cup [\gamma^{2}]}\rho^{1}_{j}(h)=12(\rho^{1}_{j}(\gamma)+\rho^{1}_{j}(\gamma^{2})).
\end{align}
Then
\begin{align}
\eta_{([1],\rho^{1}_{j})}(a,b)=\rho^{1}_{j}(1)+\sum\limits_{\{2,3,5\}\ni p|a}\omega_{j}(p). \label{eq:eta}
\end{align}
To see this, note that $\forall h\in\Gamma$, $|h|\in\{1,2,3,5\}$. Suppose $(a,60)=d$, choose an integer $c$ such that
$ac\equiv d\pmod{60}$, as in Remark 3.1. If $2|d$ and $3,5\nmid d$, then
\begin{align*}
\eta_{([1],\rho^{1}_{j})}(a,b)&=\sum\limits_{h^{d}=1}\rho^{1}_{j}(h^{-bc})=\rho^{1}_{j}(1)+\sum\limits_{h\in [\alpha]}\rho^{1}_{j}(h^{-bc})\\&=\rho^{1}_{j}(1)+\sum\limits_{h\in [\alpha]}\rho^{1}_{j}(h)=\rho^{1}_{j}(1)+\omega_{j}(2).
\end{align*}
The other cases are dealt with similarly.

The other values of $\eta_{i}(a,b)$ can be worked out more easily:
\begin{align}
&\eta_{([\alpha],\rho^{\alpha}_{j})}(a,b)=\delta_{1,a\pmod2}\cdot 15\cdot(-1)^{b(1-j)},
\hspace{4mm} 1\leqslant j\leqslant 4,\\
&\eta_{([\beta],\rho^{\beta}_{j})}(a,b)=\delta_{1,a^{2}\pmod3}\cdot 20\cdot(e^{2\pi i/3})^{ab(1-j)},
\hspace{4mm} 1\leqslant j\leqslant 3,\\
&\eta_{([\gamma],\rho^{\gamma}_{j})}(a,b)=\delta_{1,a^{4}\pmod5}\cdot 12\cdot(e^{2\pi i/5})^{a^{3}b(1-j)},
\hspace{4mm} 1\leqslant j\leqslant 5,\\
&\eta_{([\gamma^{2}],\rho^{\gamma^{2}}_{j})}(a,b)=\delta_{1,a^{4}\pmod5}\cdot 12\cdot(e^{2\pi i/5})^{2a^{3}b(1-j)}
\hspace{4mm}
1\leqslant j\leqslant 5.
\end{align}

\subsection{The result}

The values of $\dim\chi_{i}$ for $i\in\Lambda$ are:
\begin{align}
(\dim\chi_{([1],\rho^{1}_{1})},\dim\chi_{([1],\rho^{1}_{2})},\dim &\chi_{([1],\rho^{1}_{3})},
\dim\chi_{([1],\rho^{1}_{4})},\dim\chi_{([1],\rho^{1}_{5})})=(1, 3, 3, 4, 5); \\
\dim&\chi_{([\alpha],\rho^{\alpha}_{j})}=15,\hspace{4mm} 1\leqslant j\leqslant 4;\\
\dim&\chi_{([\beta],\rho^{\beta}_{j})}=20, \hspace{4mm}1\leqslant j\leqslant 3; \\
\dim\chi_{([\gamma],\rho^{\alpha}_{j})}&=\dim\chi_{([\gamma^{2}],\rho^{\gamma^{2}}_{j})}=12,
\hspace{4mm} 1\leqslant j\leqslant 5.
\end{align}
Below we simplify $\pmod p$ into $(p)$.
\begin{align}
&\sum\limits_{j=1}^{4}\frac{(\#\Gamma)^{2g-2}}{(\dim\chi_{([\alpha],\rho^{\alpha}_{j})})^{n+2g-2}}\prod\limits_{k=1}^{n}
\eta_{([\alpha],\rho^{\alpha}_{j})}(a_{k},b_{k})\nonumber\\
=&\sum\limits_{j=1}^{4}4^{2g-2}\left(\prod\limits_{k=1}^{n}\delta_{1,a_{k}(2)}\right)
(-1)^{(j-1)\sum\limits_{k=1}^{n}b_{k}}\nonumber \\
=&4^{2g-1}\delta_{0,\sum\limits_{k=1}^{n}b_{k}(2)}\prod\limits_{k=1}^{n}\delta_{1,a_{k}(2)}. \label{eq:first term}
\end{align}
Similarly,
\begin{align}
\sum\limits_{j=1}^{3}\frac{(\#\Gamma)^{2g-2}}{(\dim\chi_{([\beta],\rho^{\beta}_{j})})^{n+2g-2}}\prod\limits_{k=1}^{n}
\eta_{([\beta],\rho^{\beta}_{j})}(a_{k},b_{k})=3^{2g-1}\delta_{0,\sum\limits_{k=1}^{n}a_{k}b_{k}(3)}
\prod\limits_{k=1}^{n}\delta_{1,a^{2}_{k}(3)},
\end{align}
\begin{align}
\sum\limits_{j=1}^{5}\frac{(\#\Gamma)^{2g-2}}{(\dim\chi_{([\gamma],\rho^{\gamma}_{j})})^{n+2g-2}}
\prod\limits_{k=1}^{n}\eta_{([\gamma],\rho^{\gamma}_{j})}(a_{k},b_{k})=5^{2g-1}\delta_{0,\sum\limits_{k=1}^{n}
a^{3}_{k}b_{k}(5)}\prod\limits_{k=1}^{n}\delta_{1,a^{4}_{k}(5)},
\end{align}
\begin{align}
\sum\limits_{j=1}^{5}\frac{(\#\Gamma)^{2g-2}}{(\dim\chi_{([\gamma^{2}],\rho^{\gamma^{2}}_{j})})^{n+2g-2}}
\prod\limits_{k=1}^{n}\eta_{([\gamma^{2}],\rho^{\gamma^{2}}_{j})}(a_{k},b_{k})=5^{2g-1}
\delta_{0,\sum\limits_{k=1}^{n}a^{3}_{k}b_{k}(5)}\prod\limits_{k=1}^{n}\delta_{1,a^{4}_{k}(5)}.
\end{align}
Using (\ref{eq:eta}) to compute the last term, and noting $\omega_{2}(p)=\omega_{3}(p)$, we have
\begin{align}
&\sum\limits_{j=1}^{5}\frac{(\#\Gamma)^{2g-2}}{(\dim\chi_{([1],\rho^{1}_{j})})^{n+2g-2}}\prod\limits_{k=1}^{n}
\eta_{([1],\rho^{1}_{j})}(a_{k},b_{k})=  \nonumber \\
&60^{2g-2}\prod\limits_{k=1}^{n}\left(1+\sum\limits_{\{2,3,5\}\ni p|a_{k}}\omega_{1}(p)\right)+2\cdot 20^{2g-2}\prod\limits_{k=1}^{n}\left(1+\frac{1}{3}\sum\limits_{\{2,3,5\}\ni p|a_{k}}\omega_{2}(p)\right)+\nonumber \\ &15^{2g-2}\prod\limits_{k=1}^{n}\left(1+\frac{1}{4}\sum\limits_{\{2,3,5\}\ni p|a_{k}}\omega_{4}(p)\right)+12^{2g-2}\prod\limits_{k=1}^{n}\left(1+\frac{1}{5}\sum\limits_{\{2,3,5\}\ni p|a_{k}}\omega_{5}(p)\right).
\label{eq:last term}
\end{align}

Finally, $Z(M_{O}(g;(a_{1},b_{1}),\cdots,(a_{n},b_{n})))$ is the sum of the expressions on the right-hand side of
(\ref{eq:first term})-(\ref{eq:last term}).
Multiplying by 60 gives the number of regular $A_{5}$-coverings of the Seifert 3-manifold $M_{O}(g;(a_{1},b_{1}),\cdots,(a_{n},b_{n}))$.

\section{Further remarks}

\begin{enumerate}
\item As the last section illustrates, we can compute explicitly the number of regular $\Gamma$-coverings of any
closed orientable Seifert 3-manifolds, as long as we know enough about $\Gamma$.

\hspace{4mm} We choose $\Gamma$ to be $A_{5}$ in the example, because it is a non-solvable finite group, whence
beyond the scope of \cite{count}.

\item In \cite{enum} the authors did not deal with exceptional fibers, because they would present additional difficulties
in their approach. In our approach they are easy to deal with, thanks to the cut-and-glue property of TQFT.

\item In principle the same method can be used for computing regular $\Gamma$-coverings of general graph
3-manifolds, which can be obtained by gluing Seifert 3-manifolds along boundary tori. In that case the
$s$-matrix, which can be complicated, will play a key role, because the gluing of the tori depends on
the mapping class group of $\Sigma_{1}$.
\end{enumerate}

\end{document}